\documentclass[12pt,oneside]{article}
\pdfoutput=1
\usepackage{amsthm}
\usepackage{amsmath}
\usepackage{amssymb}
\usepackage{mathrsfs}
\usepackage{graphicx}
\usepackage{amssymb}
\usepackage{tikz}
\usepackage[normalem]{ulem}
\usepackage{authblk}
\usepackage{amsmath,amsthm,amscd,amssymb}
\usepackage{latexsym,epsfig,enumerate,multicol,wasysym}
\usepackage[colorlinks,citecolor=red,hypertexnames=false]{hyperref}
\hypersetup{
  colorlinks=true,
  citecolor=red,
  linkcolor=blue
}
\usepackage[left=0.8in,right=0.8in]{geometry}
\usepackage[backend=biber, style=numeric]{biblatex}
\usepackage[title]{appendix}

\allowdisplaybreaks

\DeclareFieldFormat*{title}{#1}
\renewbibmacro{in:}{}

\makeatletter
\newcommand{\subjclass}[2][1991]{
  \let\@oldtitle\@title
  \gdef\@title{\@oldtitle\footnotetext{#1 \emph{Mathematics subject classification.} #2}}}

\newtheorem{Definition}{Definition}
\newtheorem{Lemma}{Lemma}
\newtheorem{Proposition}{Proposition}
\newtheorem{Remark}{Remark}
\newtheorem{Theorem}{Theorem}
\newtheorem*{Theorem*}{Theorem}

\title{The existence of a bounded linear extension operator for $L^{s,p}(\mathbb{R}^n)$ when $\frac{n}{p}<\{s\}$}

\author{Han Li}
\date{}
 \AtEndDocument{%
  \par
  \medskip
  \begin{tabular}{@{}l@{}}%
    \textsc{CUNY Graduate Center Department of Mathematics, New York, NY, 10016.}\\
    \textit{Email address}: \texttt{hli@gradcenter.cuny.edu}
  \end{tabular}}

\begin{document}
\setlength{\abovedisplayskip}{5pt}
\setlength{\belowdisplayskip}{5pt}
 
\maketitle
\begin{abstract}
Let $L^{s,p}(\mathbb{R}^n)$ denote the homogeneous Sobolev-Slobodeckij space. In this paper, we demonstrate the existence of a bounded linear extension operator from the jet space $J^{\lfloor s \rfloor}_E L^{s,p}(\mathbb{R}^n)$ to $L^{s,p}(\mathbb{R}^n)$ for any $E \subseteq \mathbb{R}^n$, $p \in [1, \infty)$, and $s \in (0, \infty)$ satisfying $\frac{n}{p} < \{s\}$, where $\{s\}$ represents the fractional part of $s$. Our approach builds upon the classical Whitney extension operator and uses the method of exponentially decreasing paths.
\end{abstract}

\section{Introduction}

Let $\mathbb{X}$ be a seminormed vector space   consisting of functions that are  $m$-times differentiable, where $m\in\mathbb{N}$, and these  functions are defined on $\mathbb{R}^n$.  For any subset $E\subseteq \mathbb{R}^n$,  define the   {\it jet space}  $J^m_E\mathbb{X}=\{(J^m_xF)_{x\in E}:F\in\mathbb{X}\}$, where $J^m_xF$ represents the $m$-th Taylor polynomial of $F \in \mathbb{X}$ at the point $x \in \mathbb{R}^n$.  We refer to $(J^m_xF)_{x\in E}$ as the {\it jet} of $F$ on the set $E$. For any given jet $f$ in $J^m_E\mathbb{X}$, the seminorm of $f$ on $J^m_E\mathbb{X}$ is defined by
\begin{equation*}
\|f\|_{J^m_E\mathbb{X}}:=\inf\{\|F\|_{\mathbb{X}}:F\in\mathbb{X} \text{ and } (J^m_xF)_{x\in E}=f\}. 
\end{equation*}
Providing the $m$-th Taylor polynomials is equivalent to specifying the function values and derivatives up to $m$-th order. 

For any operator $T:J^m_E\mathbb{X}\rightarrow\mathbb{X}$, we say $T$ is an extension operator if for any $f\in J^m_E\mathbb{X}$, $(J^m_xTf)_{x\in E}=f$. This gives rise to the extension problem: {\it Can we find a bounded linear extension operator from $J^m_E\mathbb{X}$ to $\mathbb{X}$?}

The extension problem was first introduced and studied by H. Whitney in 1934 \cite{Whitney1934Analytic,Whitney1934Differentiable,Whitney1934Functions}. In particular, he developed the Whitney extension operator and provided a solution to the extension problem for $C^m(\mathbb{R}^n)$ in \cite{Whitney1934Analytic}. Later in 1958, G. Glaeser showed that the Whitney extension operator is a bounded linear operator from the jet space $J_E^m C^{m,r}(\mathbb{R}^n)$ to the H\"older space $C^{m,r}(\mathbb{R}^n)$ for any natural number $m$ and $0<r\le 1$ \cite{Glaeser1958Etude}; (see also Chapter VI in \cite{Stein1970Singular}.) For Sobolev spaces, P. Shvartsman, in 2009,  proved the existence of a bounded linear extension operator from $J^0_EL^{1,p}(\mathbb{R}^n)$ to $L^{1,p}(\mathbb{R}^n)$ and from $J^0_EW^{1,p}(\mathbb{R}^n)$ to $W^{1,p}(\mathbb{R}^n)$ for $p>n$ by adapting the classical Whitney extension operator   \cite{Shvartsman2009Sobolev}. Later in 2017, he extended the above results to $J^{m-1}_E L^{m,p}(\mathbb{R}^n)$ and $J^{m-1}_E W^{m,p}(\mathbb{R}^n)$ where $m$ is an arbitrary positive integer and $p>n$ \cite{Shvartsman2017Whitney}. For other works relevant to the classical Whitney extension operator, see, e.g., Chapter VI in \cite{Stein1970Singular} and \cite{Frerick2016Whitney,Chang2017The}.

For the extension of function values on a set $E$, without known derivatives, H. Whitney first considered the space of $m$-th continuously differentiable functions $\mathbb{X}=C^m(\mathbb{R}^n)$ and solved the problem for $n=1$ \cite{Whitney1934Differentiable}. Then C. Fefferman solved the case for all dimensions $n$ in 2006 \cite{Fefferman2006Whitney}. For the work on homogeneous Sobolev space, in 2013, A. Israel proved the existence of a bounded linear extension operator from the trace space to the homogeneous Sobolev space $L^{2,p}(\mathbb{R}^2)$ for $p>n$, where they introduced the exponential decreasing path \cite{Israel2013A}. Later, in 2014,  C. Fefferman, A. Israel and G. Luli \cite{Fefferman2014Sobolev} extended the result and proved the existence of a bounded linear extension operator for $L^{m,p}(\mathbb{R}^n)$,  where $p\in(n,\infty)$ and $m$ is any positive integer. There are also related works for the range restriction extension problem. See, e.g., \cite{Fefferman2016Finiteness,Fefferman2017Interpolation,Fefferman2021C2} and the references therein.

\medskip 

All of the above results, {\it however}, deal with Sobolev spaces with integer regularity parameters. In this paper, we address the Whitney extension problem for the homogeneous Sobolev-Slobodeckij space $\mathbb{X}=L^{s,p}(\mathbb{R}^n)$, where  $p>n$ and   the   regularity parameter $s > 0$ is  {\it non-integer}. We will require the regularity parameter to satisfy $\frac{n}{p}<\{s\}$ in order to guarantee that the functions in $L^{s,p}(\mathbb{R}^n)$ belong to $C^{\lfloor s \rfloor}(\mathbb{R}^n)$, by the Sobolev embedding theorem. Here,  $\lfloor s\rfloor\in\mathbb{N}$ denotes the integer part of $s$,  and 
  $\{s\} \in [0,1)$  represents the fractional part.
\medskip

We note that a {\it key difference} between the current work and the related work of \cite{Shvartsman2009Sobolev,Shvartsman2017Whitney} is the consideration of {\it non-integer regularity parameter} 
$s$. While \cite{Shvartsman2009Sobolev,Shvartsman2017Whitney} rely on the estimation of sharp maximal functions, this work adopts an alternative approach, using the technique of integrating along exponentially decreasing paths in the Whitney decomposition.

\medskip 

Our main result can be stated as follows: 

\begin{Theorem}[Main Result]\label{main}
Let $n \in \mathbb{N}^*$, $s \in (0, \infty)$, and $p \in [1, \infty)$. If $\frac{n}{p} < \{s\}$, where $\{s\}$ is the fractional part of $s$, and $E \subseteq \mathbb{R}^n$, then there exists a bounded linear extension operator $T : J^{\lfloor s \rfloor}_E L^{s,p}(\mathbb{R}^n) \to L^{s,p}(\mathbb{R}^n)$ such that
\begin{equation*}
\|Tf\|_{L^{s,p}(\mathbb{R}^n)} \le C_{n,s,p} \|f\|_{J^{\lfloor s \rfloor}_E L^{s,p}(\mathbb{R}^n)},
\end{equation*}
where $C_{n,s,p}$ is a positive constant depending only on $n$, $s$, and $p$.
\end{Theorem}

{\it Technique of the proof:} To establish the result, we construct our operator by adapting techniques from the classical Whitney extension operator (see, for example, Chapter \text{VI} in \cite{Stein1970Singular}), and utilize the exponentially decreasing paths (see Lemma \ref{2}) for norm estimation. Specifically, we prove the theorem through the following steps. {\it First}, we decompose $\overline{E}^c$, which is   the complement of the closure of set $E$,   into a family of cubes $\mathscr{F}= \{Q_k\}$ by adapting the classical Whitney decomposition, and for each cube $P$ in $\mathscr{F}$ we choose a point $x_P\in \overline{E}$ which is also one of the closest points in $\overline{E}$ to the cube $P$ (and it might not be unique). {\it Then} we construct the extension operator $T$ by the classical method and show it is indeed bounded.

To demonstrate that  $T$ is bounded, 
we need to show that 
\begin{equation*}
\|Tf\|_{L^{s,p}(\mathbb{R}^n)}\lesssim\|F\|_{L^{s,p}(\mathbb{R}^n)},
\end{equation*}
for any $f$ and 
for any $F\in L^{s,p}(\mathbb{R}^n)$ satisfying $(J_x^{\lfloor s\rfloor}F)_{x\in E}=f$. 
  
For this purpose, we  decompose $\|Tf\|_{L^{s,p}(\mathbb{R}^n)}$ into several pieces and show that each piece is bounded by the sum of $\|F\|_{L^{s,p}(\mathbb{R}^n)}$ and another term whose main part is about the difference between $\partial^iF$ and its Taylor polynomial at some point $a_P$, i.e.,
\begin{equation}\label{difference}
J^{\lfloor s\rfloor-|i|}_{x_P}\partial^iF(a_P)-\partial^i F(a_P).
\end{equation}

 By the geometric properties of our Whitney decomposition (as detailed in Lemma  \ref{1} below), the line connecting $a_P$ and $x_P$ can be covered by countably many cubes $\{P_j\}_{j=0}^{\infty}$ from the decomposition $\mathscr{F}$, 
 which satisfies the exponentially decreasing property described in Lemma \ref{2}. 
 Using this approach, we bound the term  \eqref{difference} above by a series involving $\|F\|_{L^{s,p}(P_j)}$. Consequently, this provides an upper bound for $\|Tf\|_{L^{s,p}(\mathbb{R}^n)}$ in terms of $\|F\|_{L^{s,p}(\mathbb{R}^n)}$. This completes the proof of the main theorem.
\vskip.1in 

The organization of the paper is as follows: After introducing the necessary notations and preliminary results in Section \ref{notations}, we prove the main lemmas in Section \ref{main-lemmas}, which are essential for establishing the main result of this paper.  In Section \ref{proof}, we present the proof of Theorem \ref{main}. 
\vskip.1in 

{\bf Acknowledgement:}
The author would like to thank their advisor Azita Mayeli and co-advisor Arie Israel, for many helpful discussions. This research is partially supported by the CUNY grant DSRG (Round 17). 

\section{Notation and  preliminaries}\label{notations}

Let $\mathbb{N}=\{0,1,2,\dots\}$ and $\mathbb{N}^*=\{1,2,\dots\}$. For any $s\in\mathbb{R}$, let $\lfloor s\rfloor=\max\{m\in\mathbb{Z}:m\le s\}$ and $\{s\}=s-\lfloor s\rfloor$. In this paper, we always assume $n\in\mathbb{N}^*$, $s\in(0,\infty)\backslash\mathbb{N}^*$, $p\in[1,\infty)$. Let $S$ be $\mathbb{R}^n$ or any cube in $\mathbb{R}^n$ (See Definition \ref{def}), recall the homogeneous Sobolev-Slobodeckij space is given by 
\begin{equation*}
L^{s,p}(S):=
\left\{f\in L^1_{loc}(S,\mathbb{R}):\nabla^{\lfloor s\rfloor} f\in L^1_{loc}(S,\mathbb{R}^{n^{\lfloor s\rfloor}})\text{ and } \|f\|_{L^{s,p}(S)}<\infty\right\}
\end{equation*}
where the seminorm is defined as
\begin{equation}\label{e14}
\|f\|_{L^{s,p}(S)}:= \left(\int_{S}\int_{S}\frac{\big|\sum_{|k|=\lfloor s\rfloor}|\partial^kf(x)-\partial^kf(y)|\big|^p}{|x-y|^{n+\{s\}p}}dxdy\right)^{\frac{1}{p}}. 
\end{equation}
For any $m\in\mathbb{N}^*$, $F\in C^m(\mathbb{R}^n)$ and $x\in\mathbb{R}^n$, let $J^m_xF(t)=\sum_{|k|\le m}\frac{\partial^k F(x)}{k!}(t-x)^k$ which is the $m$-th Taylor polynomial of $F$ at point $x$. For any $E\subseteq\mathbb{R}^n$, if $\frac{n}{p}<\{s\}$,  we define the vector space
\begin{equation}\label{15}
J^{\lfloor s\rfloor}_EL^{s,p}(\mathbb{R}^n):=\{(J^{\lfloor s\rfloor}_xF)_{x\in E}:F\in L^{s,p}(\mathbb{R}^n)\}, 
\end{equation}
with seminorm
\begin{equation*}
\|f\|_{J^{\lfloor s\rfloor}_EL^{s,p}(\mathbb{R}^n)}:=\inf\{\|F\|_{L^{s,p}(\mathbb{R}^n)}:F\in L^{s,p}(\mathbb{R}^n) \text{ and } (J^{\lfloor s\rfloor}_xF)_{x\in E}=f\}. 
\end{equation*}

\begin{Remark}
When $\frac{n}{p} < \{s\}$, every function in $L^{s,p}(\mathbb{R}^n)$ has a unique $\lfloor s\rfloor$-th continuously differentiable representative, thanks to a Sobolev embedding theorem for $\mathbb{R}^n$ similar to that stated in Proposition \ref{5}. We will consistently use this $\lfloor s\rfloor$-th continuously differentiable function as the representative of the function class, and hence, we restrict and evaluate the representative on the set $E$.
\end{Remark}

\begin{Definition}\label{def}
We say that $Q$ is a cube in $\mathbb{R}^n$ if $Q = \prod_{i=1}^n [a_i - t, a_i + t]$, where $a_i \in \mathbb{R}$ and $t \in (0, \infty)$. If $Q = \prod_{i=1}^n [a_i - t, a_i + t]$ is a cube and $c \in (0, \infty)$, we define $cQ = \prod_{i=1}^n [a_i - ct, a_i + ct]$. For two cubes $Q$ and $Q'$, we write $Q \leftrightarrow Q'$ if $Q \cap Q' \ne \emptyset$. 
For a given cube $Q$ in $\mathbb{R}^n$, we define $\delta_Q$ as the diameter of $Q$, i.e., $\delta_Q = \sup \{ |x - y| : x, y \in Q \}$. For subsets $A$ and $B$ of $\mathbb{R}^n$, we denote the distance between $A$ and $B$ as $\text{dist}(A, B) = \inf \{ |x - y| : x \in A, y \in B \}$. For $x\in\mathbb{R}^n$ and $A\subseteq\mathbb{R}^n$,  we denote the distance between $x$ and $A$ as dist$(x,A)=\inf\{|x-y|:y\in A\}$. For any $E\subseteq\mathbb{R}^n$, $\overline{E}$, $E^\circ$ and $E^c$ denote the closure, interior and complement of $E$ respectively. For $x, y \in \mathbb{R}^n$, we let $[x, y] = \{ (1 - t)x + ty : t \in [0, 1] \}$ and $[x, y) = \{ (1 - t)x + ty : t \in [0, 1) \}$.
\end{Definition}

Let $A$ and $B$ be real. By $A\lesssim B$,  we mean  $A\le CB$,  where the $C$ is a constant  depending  only on $n,s,p$, unless  otherwise specified. Throughout the paper, the constants $C$ may vary from one line to another. 

We recall the definitions of  the H\"older spaces and the  Sobolev embedding theorem.

\begin{Definition}[Homogeneous H\"older spaces]
For $m \in \mathbb{N}$ and $0 < r < 1$, we define $\dot{C}^{m,r}(\mathbb{R}^n)$ as the space of functions $f \in C^m(\mathbb{R}^n)$ for which the following seminorm is finite: 
\begin{equation*} \|f\|_{\dot{C}^{m,r}(\mathbb{R}^n)} = \sup_{x, y \in \mathbb{R}^n, \, x \ne y} \frac{|\nabla^m f(x) - \nabla^m f(y)|}{|x - y|^r}.  \end{equation*}
 
\end{Definition}

\begin{Proposition}[Sobolev Embedding Theorem \cite{Leoni2023A}]\label{5}
Let $\frac{n}{p} < \{s\}$ and let $Q$ be a cube in $\mathbb{R}^n$. Then every $F \in L^{s,p}(Q)$ has a unique representative $\overline{F} \in \dot{C}^{\lfloor s \rfloor, \{s\} - \frac{n}{p}}(Q)$ such that
\begin{equation*} \|\overline{F}\|_{\dot{C}^{\lfloor s \rfloor, \{s\} - \frac{n}{p}}(Q)} \lesssim \|F\|_{L^{s,p}(Q)}. \end{equation*}
In particular, for any $k$ satisfying $|k|=\lfloor s\rfloor$, $x,y\in Q$, we have
\begin{equation*}
|\partial^k\overline{F}(x)-\partial^k\overline{F}(y)|\lesssim\|F\|_{L^{s,p}(Q)}|x-y|^{\{s\} - \frac{n}{p}}.
\end{equation*}
\end{Proposition}

\begin{Proposition}\label{p1}
Let $m\in\mathbb{N}$, $F\in C^m(\mathbb{R}^n)$.  For any $x,x_0\in\mathbb{R}^n$, there exists $t\in(0,1)$ such that
\begin{equation*}
F(x)-J_{x_0}^mF(x)=\sum_{|k|=m}\frac{1}{k!}\big(\partial^kF(tx+(1-t)x_0)-\partial^kF(x_0)\big)(x-x_0)^{k}.
\end{equation*}
\end{Proposition}
\begin{proof}
We first prove the result in the one-dimensional case, i.e., $n = 1$. Let $x, x_0 \in \mathbb{R}$. Then, by the Lagrange remainder formula (see \cite{Rudin1976Principles}), there exists $t \in (0,1)$ such that 
\begin{equation*}
\begin{split}
F(x)-J^m_{x_0}F(x)=&F(x)-J^{m-1}_{x_0}F(x)-\frac{1}{m!}\partial^mF(x)(x-x_0)^m\\
=&\frac{1}{m!}\big(\partial^mF(tx+(1-t)x_0)-\partial^mF(x_0)\big)(x-x_0)^m.
\end{split}
\end{equation*}

In the higher-dimensional case, let $\varphi(t) = F(tx+(1-t)x_0)$. Then $\varphi \in C^m(\mathbb{R})$, and by the one-dimensional result, there exists $t \in (0,1)$ such that 
\begin{equation}\label{e1}
\varphi(1)-J^m_{0}\varphi(1)=\frac{1}{m!}(\partial^m\varphi(t)-\partial^m\varphi(0)).
\end{equation}
For any $j\in\{0,1,\dots, m\}$ and $t\in\mathbb{R}$, by chain rule, we have
\begin{equation*}
\partial^j\varphi(t)=\left(((x-x_0)\cdot\nabla)^jF\right)(tx+(1-t)x_0)=\sum_{|k|=j}\frac{j!}{k!}\partial^kF(tx+(1-t)x_0)(x-x_0)^k.
\end{equation*}
Using the above equation, we get
\begin{equation*}
\begin{split}
J^m_{0}\varphi(1)=\sum_{j=0}^m\frac{1}{j!}\partial^j\varphi(0)
=\sum_{j=0}^m\sum_{|k|=j}\frac{1}{k!}\partial^kF(x_0)(x-x_0)^k
=\sum_{|k|\le m}\frac{1}{k!}\partial^kF(x_0)(x-x_0)^k
=J^m_{x_0}F(x),
\end{split}
\end{equation*}
and
\begin{equation*}
\partial^m\varphi(t)-\partial^m\varphi(0)=\sum_{|k|=m}\frac{m!}{k!}
(\partial^kF(tx+(1-t)x_0)-\partial^kF(x_0))(x-x_0)^k.
\end{equation*}
We complete the proof of the proposition by substituting the last two equations into equation (\ref{e1}) and using the fact that $\varphi(1) = F(x)$. 
\end{proof}
The following Lemmas \ref{1}-\ref{3} are based on Whitney decompositions, with proofs available in \cite{Stein1970Singular}. We have made slight modifications to the original proofs, which result in different constants in Lemma \ref{1}(\ref{i5}) and Lemma \ref{6} for our subsequent purposes. 

\begin{Lemma}\label{1}
Let $D$ be any nonempty closed proper subset in $\mathbb{R}^n$. Then there exists a collection of closed cubes $\mathscr{F}=\{Q_k\}_{k=1}^{\infty}$ such  that the following hold: 
\begin{enumerate}[(1)]
\item $\bigcup_{k}Q_k=D^c$.

\item The interiors of $Q_k$ are mutually disjoint.  

\item\label{i5} $10\delta_{Q_k}\le$ dist$(Q_k,D)\le 22\delta_{Q_k}$ for any $k$. \label{eq2}

\item For every $k$, there exists $m\in\mathbb{Z}$ and integers $(a_i)_{i=1}^n\in\mathbb{Z}^n$ such that $Q_k=\prod_{i=1}^n[a_i2^{-m},(a_i+1)2^{-m}]$.
\end{enumerate}
\end{Lemma}

\begin{Lemma}\label{6}
Let $Q$ and $Q'\in \mathscr{F}$ be in the decomposition outlined in Lemma \ref{1}.
\begin{enumerate}[(1)]
\item\label{i7} If $Q\leftrightarrow Q'$, then
\begin{equation*}
\frac{1}{2}\delta_{Q'}\le\delta_Q\le 2\delta_{Q'}.
\end{equation*}
\item\label{i8} If $Q\not\leftrightarrow Q'$, then $Q'\bigcap (2Q)^\circ=\emptyset$, therefore $Q'\not\leftrightarrow 1.1Q$.
\end{enumerate}
\end{Lemma}

\begin{Lemma}\label{14} 
There exists a constant $C$, dependent solely on the dimension $n$, such that for every cube $Q \in \mathscr{F}$ in the decomposition specified in Lemma \ref{1}, there are at most $C$ cubes that intersect $Q$.
\end{Lemma}

\begin{Lemma}\label{9}
There exists a constant $C$, dependent solely on the dimension $n$, such that any point in $D^c$ lies within a neighborhood that intersects at most $C$ cubes of the form $1.1Q$, where $Q \in \mathscr{F}$.
\end{Lemma}

\begin{Lemma}\label{3} For a given decomposition $\mathscr{F}$ outlined in Lemma \ref{1}, there is a family of functions $\{\theta_Q\}_{Q\in \mathscr F}\subseteq C^\infty(\mathbb{R}^n)$,  with the following properties:
\begin{enumerate}[(1)]
\item $0\le \theta_Q(x)\le 1$ for all    $x\in \mathbb{R}^n$.
\item\label{i11} $\theta_Q(x)=0$ outside $1.1Q$.
\item\label{i4} $\sum_Q\theta_Q(x)=1$ for all $x\in D^c$. And 
\item\label{i6}  
$ |\partial^k\theta_Q(x)| \leq C_{n,k} \delta_Q^{-|k|} \text{ for all } Q \in \mathscr{F}, \, k \in \mathbb{N}^n, \text{ and } x \in \mathbb{R}^n,$ 
where $C_{n,k}$ is a constant that depends only on $n$ and $k$.
\end{enumerate}
\end{Lemma}

\begin{Remark}\label{11}
By Lemma \ref{3} (\ref{i6}), for any $s > 0$, $|k| \le \lfloor s \rfloor + 1$, $Q \in \mathscr{F}$, and $x \in \mathbb{R}^n$, we have

\begin{equation*}
|\partial^k\theta_Q(x)| \leq\big(\max_{|j|\le\lfloor s \rfloor + 1}C'_{n,j}\big)\delta_Q^{-|k|}=C_{n,s} \delta_Q^{-|k|},
\end{equation*}
where the constant $C_{n,s}=\max_{|j|\le\lfloor s \rfloor + 1}C'_{n,j}$ depends only on $n$ and $s$.
\end{Remark}

We have the following definition and lemma.

\begin{Definition}\label{24}
Associated to any cube $P\in\mathscr{F}$ in the decomposition specified in Lemma \ref{1}, let $x_P\in D$ such that $dist(P,D)=dist(x_P,P)$ and fix it. (Such $x_P$ exists because $D$ is closed.)
\end{Definition}

\begin{Lemma}\label{10}
Let $D$ be a closed subset of $\mathbb{R}^n$ and let $\mathscr{F}$ be a decomposition outlined in Lemma \ref{1} with respect to $D$. Let $(\theta_P)_{P\in\mathscr{F}}$ be a family of functions satisfying the conditions in Lemma \ref{3}. For any $P\in\mathscr{F}$, let  $x_P$ be defined in Definition \ref{24}. Then for any $m\in\mathbb{N}$ and $F\in C^m(\mathbb{R}^n)$, the following function
\begin{equation*}
G(x)=\begin{cases}
\sum_{P\in\mathscr{F}}\theta_P(x)J_{x_P}^{m}F(x),&x\in\mathbb{R}^n\backslash D;\\
F(x),&x\in D.
\end{cases}
\end{equation*}
is well defined. $G\in C^m(\mathbb{R}^n)$ and $\partial^i G(x)=\partial^i F(x)$ for any $x\in D$ and for all multi-indices 
  $i$ satisfying $|i|\le m$.
\end{Lemma}

\begin{proof}
The theorem follows the proof of Lemma 2 in \cite{Whitney1934Analytic} with slightly different constants. The proof can also be found in Appendix \ref{A1}.
\end{proof}

\section{Main Lemmas}\label{main-lemmas}

The proof of our main result in Theorem \ref{main} uses three main lemmas, which we will prove here. For them we let $D$ be a closed nonempty proper subset of $\mathbb{R}^n$, and let $\mathscr{F}$ be the family of cubes as in Lemma \ref{1}. First, we have a definition.

\begin{Definition}\label{defn:paths}
Associated to any  cube $P\in\mathscr{F}$, let  $A_P$ denote the collection of paths defined as 
\begin{align*}
A_P=\{tx_1+(1-t)x_2:  ~ & Q_1,Q_2,Q_1',Q_2'\in\mathscr{F},\,t\in[0,1], \\
&\,x_1\in Q_1\leftrightarrow Q_2\leftrightarrow P\leftrightarrow Q_2'\leftrightarrow Q_1'\ni x_2\}.
\end{align*}
\end{Definition}

\begin{Lemma}[Exponentially decreasing path]\label{2}
Let $\mathscr{F}$ be a decomposition outlined in Lemma \ref{1}.  Let $P\in\mathscr{F}$. Then there exists $A>0$ and $a\in(0,1)$ which are only depending on $n$ such that for any $x\in A_P$ (see Definition \ref{defn:paths}), there exists a family of cubes $\{P_j = P_j(P,x)\}_{j=0}^\infty$ in $\mathscr{F}$ such that
\begin{enumerate}[(1)]
\item\label{i9} $[x,x_P)\subseteq\bigcup_{j\in\mathbb{N}} P_j$. ($x_p$ is defined in Definition \ref{24}.)
\item\label{i10} $x\in P_0$, $P_j\leftrightarrow P_{j+1}$ for all $j\geq 0$.
\item\label{i1} For any $i,j$ satisfying $i\le j$, $\delta_{P_j}\le Aa^{j-i}\delta_{P_i}$.
\end{enumerate}
\end{Lemma}
  
\begin{proof}
Let $c_P\in P$ such that 
\begin{equation}\label{e18}
|x_P-c_P|=\text{dist}(x_P,P)=\text{dist}(P,D).
\end{equation}
(Such $c_P$ exists because $P$ is closed.) 

{\it Claim 1:}\label{16}
For any $x\in A_P$, the interval $[x,x_p)\subseteq D^c$.

To prove the claim  we proceed as follows:

Given any $x \in A_P$, there exist points $x_1, x_2 \in A_P$ and cubes $Q_1, Q_2, Q_1', Q_2' \in \mathscr{F}$, and a parameter $t \in [0,1]$ such that
\begin{equation*}
x = tx_1 + (1-t)x_2
\end{equation*}
where the sequence of sets satisfies:
\begin{equation*}
x_1\in Q_2 \leftrightarrow Q_1 \leftrightarrow P \leftrightarrow Q_2' \leftrightarrow Q_1'\ni x_2.
\end{equation*}
Let $a_1\in Q_2\bigcap Q_1$ and $a_2\in Q_1\bigcap P$, by applying triangle inequality and Lemma \ref{6} (\ref{i7}), we have
\begin{equation*}
|x_1-c_P|\le|x_1-a_1|+|a_1-a_2|+|a_2-c_P|\le \delta_{Q_2} + \delta_{Q_1} + \delta_P\le 4\delta_P + 2\delta_P + \delta_P = 7\delta_P.
\end{equation*}
Similarly,
\begin{equation*}
|x_2 - c_P| \le 7\delta_P.
\end{equation*}

To find the distance from $x$ to $c_P$, we use:
\begin{equation*}
|x-c_P|= |tx_1 + (1-t)x_2 - c_P|.
\end{equation*}
Applying the triangle inequality, we get:
\begin{equation*}
|tx_1 + (1-t)x_2 - c_P|=|t(x_1-c_P) + (1-t)(x_2- c_P)| \le t|x_1 - c_P| + (1-t)|x_2 - c_P|.
\end{equation*}
Substituting the bounds we obtained:
\begin{equation*}
|x-c_P| \le t \cdot 7\delta_P + (1-t) \cdot 7\delta_P = 7\delta_P.
\end{equation*}
We know that:
\begin{equation*}
7\delta_P \le \frac{7}{10}\text{dist}(P, D),
\end{equation*}
because $P$ is chosen such that $\delta_P \le \frac{1}{10} \text{dist}(P, D)$. 
Hence,
\begin{equation}\label{e19}
|x-c_P| \le \frac{7}{10}\text{dist}(P, D).
\end{equation}

To prove the claim, we continue as follows: 
For any $x' \in [x, x_P)$, there exists $t \in [0,1)$ such that
\begin{equation*}
x' = tx_P + (1-t)x.
\end{equation*}
We need to show that $x' \in D^c$, where $D^c$ denotes the complement of $D$. To do this, we first calculate the distance from $x'$ to $c_P$ and use it to show that $x'$ cannot be in $D$.

Consider the distance from $x'$ to $c_P$:
\begin{equation*}
|x' - c_P| = |tx_P + (1-t)x - c_P|.
\end{equation*}
Using the triangle inequality, we get:
\begin{equation*}
|tx_P + (1-t)x - c_P| \le t|x_P - c_P| + (1-t)|x - c_P|.
\end{equation*}
Given that $|x_P-c_P| = \text{dist}(P, D)$ and applying the given bounds and inequality (\ref{e19}), we have:
\begin{equation*}
|x'-c_P| \le t|x_P - c_P| + (1-t)|x - c_P| \le \left(t + \frac{7}{10}(1-t)\right)\text{dist}(P, D).
\end{equation*}
We know that $t + \frac{7}{10}(1-t) < 1$ because $\frac{7}{10} < 1$ and $1-t>0$. Therefore,
\begin{equation*}
|x'-c_P| < \text{dist}(P, D).
\end{equation*}
Since $|x'-c_P|$ is strictly less than $\text{dist}(P, D)$, it follows that:
\begin{equation*}
\text{dist}(x', D) \ge \text{dist}(c_P, D) - |c_P - x'| > \text{dist}(P, D) - \text{dist}(P, D) = 0.
\end{equation*}
Thus, $\text{dist}(x', D) > 0$, implying that $x' \not\in D$ and therefore $x' \in D^c$.
This completes the proof of  Claim 1. 
 
\vskip.1in 

{\it Claim 2.}\label{18}
We claim  that for any $Q \in \mathscr{F}$, if $x' \in Q \cap [x, x_P)$, then 
\begin{equation}\label{e2}
10\delta_Q \le |x' - x_P| \le 131\delta_Q.
\end{equation}

For the left inequality, we have
\begin{equation*}
10\delta_Q \le \text{dist}(Q, D) \le |x'- x_P|.
\end{equation*}

For the right inequality, let $t \in (0, 1]$ such that
\begin{equation}\label{e20}
x' - x_P = t(x - x_P).
\end{equation}
Then, by equation (\ref{e18}) and inequality (\ref{e19}), we can write:  
\begin{equation*}
|c_P - x'| = |(1 - t)(c_P - x_P) + t(c_P - x)| \le \left(1 - \frac{3t}{10}\right)\text{dist}(P, D).
\end{equation*}
There exists $a \in Q$ and $b \in D$ such that $|a- b| = \text{dist}(Q, D)$. We know that 
\begin{equation*}
\text{dist}(P,D)\le|c_P-b|
\end{equation*}
because $c_P\in P$ and $b\in D$.
By Lemma \ref{1}(\ref{i5}),  triangle inequality and the above two inequalities, we  obtain 
\begin{equation*}
23\delta_Q \ge \delta_Q + \text{dist}(Q, D) = \delta_Q + |a - b| \ge |x' - b| \ge |c_P - b| - |c_P - x'| \ge \frac{3}{10}t \, \text{dist}(P, D).
\end{equation*}
By triangle inequality, inequality (\ref{e19}) and equation (\ref{e18}),
\begin{equation}\label{e24}
|x - x_P| \le |x - c_P| + |c_P - x_P| \le \frac{17}{10}\text{dist}(P, D).
\end{equation}
Therefore, by equation (\ref{e20}), and the above two inequalities, we get
\begin{equation*}
|x'-x_P| = t|x-x_P| \le \frac{17}{10}t \, \text{dist}(P, D) \le\frac{17}{3}\cdot23\delta_Q\le 131\delta_Q.
\end{equation*}
This completes the proof of Claim 2.
\vskip.1in

{\it Claim 3:} \label{17}
There is a subfamily $\mathscr{U}_0$ of $\mathscr{F}$ such that 
\begin{equation*}
[x,x_P)\subset \bigcup_{R\in\mathscr{U}_0}R,
\end{equation*}
and for any $Q \in \mathscr{U}_0$,
\begin{equation*}
[x, x_P) \not\subseteq \bigcup_{R\in\mathscr{U}_0,R\ne Q}R,
\end{equation*}
i.e. the segment $[x, x_P)$ is contained in the union of the elements of $\mathscr{U}_0$ but is not completely contained in the union of the elements of $\mathscr{U}_0$ excluding $Q$ for any $Q\in\mathscr{U}_0$.

To prove this, let us define the set $\mathbb{A} = \{\mathscr{U} \subseteq \mathscr{F}: [x, x_P) \subseteq \bigcup_{R\in\mathscr{U}}R \}$.

For any $\mathscr{U}, \mathscr{V} \in \mathbb{A}$, we define $\mathscr{U} \le \mathscr{V}$ if $\mathscr{U} \supseteq \mathscr{V}$. With this definition, $\mathbb{A}$ is a partially ordered set. Our goal is to use Zorn's lemma to show that the set $\mathbb{A}$ has a maximum element $\mathscr{U}_0$, for which the hypothesis of our claim holds.

Note that $\mathbb{A}$ is nonempty because $[x, x_P) \subseteq D^c$ and $\mathscr{F} \in \mathbb{A}$. Let $\mathbb{B}$ be any nonempty chain in $\mathbb{A}$, and let $\mathscr{W} = \bigcap_{\mathscr{U}\in\mathbb{B}}\mathscr{U}$. Then $\mathscr{W} \in \mathbb{A}$. Indeed, by definition, $\mathscr{W} \subseteq \mathscr{F}$. To show that $\mathscr{W}$ belongs to $\mathbb{A}$, we only need to show that $[x, x_P) \subseteq \bigcup_{R\in\mathscr{W}}R$.

For any $x' \in [x, x_P)$, we claim  that there exists $R \in \mathscr{F}$ such that $x' \in R$ and for all $\mathscr{U} \in \mathbb{B}$, $R \in \mathscr{U}$. Thus, $x' \in R \in \mathscr{W}$. We prove the claim by contradiction.

If the above statement is false, then for any $Q$ satisfying $x'\in Q$, there exists a $\mathscr{U}_Q\in\mathbb{B}$ such that $Q\notin\mathscr{U}_Q$. Because the set $\{ Q \in \mathscr{F}: x'\in Q\}$ is finite (see Lemma \ref{9}) and $\mathbb{B}$ is a chain, there exists $Q'\in\{ Q \in \mathscr{F}: x'\in Q\}$ such that $\mathscr{U}_Q\le\mathscr{U}_{Q'}$ for any $Q\in\{ Q \in \mathscr{F}: x'\in Q\}$. Then for any $Q\in\{ Q \in \mathscr{F}: x'\in Q\}$, $\mathscr{U}_{Q'}\subset\mathscr{U}_Q$. By this and the fact $Q\notin\mathscr{U}_Q$, we have $Q\notin\mathscr{U}_{Q'}$. We get that $x'\notin\bigcup_{R\in\mathscr{U}_{Q'}}R$. So $[x,x_P)\not\subseteq\bigcup_{R\in\mathscr{U}_{Q'}}R$, therefore $\mathscr{U}_{Q'}\notin\mathbb{A}$. This contradicts the fact $\mathscr{U}_{Q'}\in\mathbb{B}\subseteq\mathbb{A}$, and completes the proof of the claim.

Hence,  $[x, x_P) \subseteq \bigcup_{R\in\mathscr{W}}R$. Therefore, $\mathscr{W}\in \mathbb{A}$, as desired.

Notice that $\mathscr{W}$ is an upper bound of $\mathbb{B}$ by its definition. Therefore, by Zorn's lemma, $\mathbb{A}$ has a maximum element $\mathscr{U}_0$, which means that $[x,x_P)\subset \bigcup_{R\in\mathscr{U}_0}R$, and for any $Q \in \mathscr{U}_0$, $[x, x_P) \not\subseteq \bigcup_{R\in\mathscr{U}_0,R\ne Q}R$. This completes the proof of Claim 3.

\vskip.1in 
 
Let $Q \in \mathscr{U}_0$.
We define $a_Q$ and $b_Q$ as the minimum and maximum values, respectively, of
$t \in [0,1)$ such that $x + t(x_P - x) \in Q$, where $x\in A_p$ and  it is  fixed at the beginning of Claim \hyperref[16]{1}. (Notice that the minimum and maximum values exist because $Q$ is closed.) By the convexity of the cube $Q$, we can write 
\begin{equation*} 
[a_Q, b_Q] = \{ t \in [0, 1) : x + t(x_P - x) \in Q \}.
\end{equation*}
For every $t \in (0, 1)$, as stated in  inequality (\ref{e2}), for any $Q\in\mathscr{U}_0$ satisfying $a_Q<t$, we have 
\begin{equation*}
\frac{1}{131}|(1-t)(x-x_P)|\le\frac{1}{131}|(1-a_Q)(x-x_P)|=\frac{1}{131}|(x+a_Q(x_P-x))-x_P|\le\delta_Q.
\end{equation*}
There are only finitely many $Q\in\mathscr{U}_0$ such that $a_Q < t$ because of the above lower bound of $\delta_Q$. Note that for any $Q,Q'\in\mathscr{U}_0$, if $Q\ne Q'$ ,then $a_Q\ne a_{Q'}$. (Otherwise either $[a_Q,b_Q]\subseteq[a_{Q'},b_{Q'}]$ or $[a_{Q'},b_{Q'}]\subseteq[a_Q,b_Q]$, this contradicts with that $\mathscr{U}_0$ is a maximum element which is shown in Claim \hyperref[17]{3}.)
Therefore, there is a bijection from $\mathbb{N}$ to $\mathscr{U}_0$, denoted by $j \mapsto P_j$, such that for all $j < k$, $a_{P_j} < a_{P_k}$. This way we have constructed a family  $\{P_j\}_{j=0}^{\infty}$  of cubes in $\mathscr{F}$, as mentioned in the statement of the lemma. 
\vskip.1in

{\it Proof of property (\ref{i9}):} Because $\{P_j\}_{j=0}^\infty$ is an enumeration of $\mathscr{U}_0$, and because $\mathscr{U}_0\in\mathbb{A}$, by the definition of $\mathbb{A}$, we have
\begin{equation*}
[x,x_P)\subseteq\bigcup_{R\in\mathscr{U}_0}R=\bigcup_{j=0}^{\infty} P_j. 
\end{equation*}
\vskip.1in

{\it Proof of property (\ref{i10}):} We will show for all $j$,  $P_j \leftrightarrow P_{j+1}$. 
Otherwise, $b_{P_j} < a_{P_{j+1}}$,  and there exists an $i$ such that $P_i \cap \{ x + t(x_P - x) : t \in (b_{P_j}, a_{P_{j+1}}) \} \neq \emptyset$. Therefore, $[a_{P_i}, b_{P_i}] \cap (b_{P_j}, a_{P_{j+1}}) \neq \emptyset$. This implies that $a_{P_i} < a_{P_{j+1}}$, which means $a_{P_i} \leq a_{P_j} \leq b_{P_j} \leq b_{P_i}$. Hence, $[x, x_P) \subseteq \bigcup_{R\in\mathscr{U}_0,R\ne P_j}R$, which contradicts the fact that $\mathscr{U}_0$ is a maximum element.
\vskip.1in

{\it Proof of property (\ref{i1}):}
We know that, by Lemma \ref{6}(\ref{i7}) and property (\ref{i10}), there exists a strictly increasing function $g : \mathbb{N} \to \mathbb{N}$ such that $g(0) = 0$ and
\begin{equation}\label{e22}
\delta_{P_{g(m)}} = \frac{1}{2^m} \delta_{P_0}.
\end{equation}
For any $m$, for any $i$ satisfying $g(m) \le i \le g(m+1)$, let $x_m = x + a_{g(m)} (x_P - x)$ and $x_{m+1} = x + b_{g(m+1)} (x_P - x)$. We have $P_i\bigcap[x_m,x_{m+1}]\ne\emptyset$ because of the choice of $i$. Therefore there exists $x_i \in P_i \cap [x_m, x_{m+1}]$. Therefore we have
\begin{equation*}
|x_{m+1}-x_P|\le|x_i-x_P|
\end{equation*}
because $x_{m+1}$ is in the middle of the segment connecting $x_i$ and $x_P$. Similarly, we have
\begin{equation*}
|x_i-x_P|\le|x_m-x_P|.
\end{equation*}
Therefore, by equation (\ref{e22}), Lemma \ref{1}(\ref{i5}), Claim \hyperref[18]{2} and the above two inequalities, we get lower bound and upper bound for $\delta_{P_i}$:
\begin{equation}\label{e21}
\begin{split}
\frac{5}{131}\delta_{P_{g(m)}}=\frac{10}{131}\delta_{P_{g(m+1)}}\le&\frac{1}{131}\text{dist}(P_{g(m+1)},D) \le \frac{1}{131}|x_{m+1}-x_P| \le \frac{1}{131}|x_i-x_P|\\ 
\le&\delta_{P_i}\le \frac{1}{10}|x_i-x_P| \le \frac{1}{10}|x_m-x_P| \le \frac{131}{10}\delta_{P_{g(m)}}.
\end{split}
\end{equation}
So for any $x'\in P_i$,
\begin{equation*}
|x'-x_i|\le\delta_{P_i}\le\frac{131}{10}\delta_{P_{g(m)}}.
\end{equation*}
By Claim \hyperref[18]{2}, we have
\begin{equation*}
|x_m - x_i| \le |x_m - x_P| \le 131\delta_{P_{g(m)}}.
\end{equation*}
So by triangle inequality and combining the above two inequalities, we get
\begin{equation*}
|x'-x_m|\le|x'-x_i|+|x_i-x_m|\le(\frac{131}{10}+131)\delta_{P_{g(m)}}<145\delta_{P_{g(m)}}.
\end{equation*}
Therefore
\begin{equation*}
P_i \subseteq B(x_m, 145\delta_{P_{g(m)}}).
\end{equation*}
Hence, by this, the lower bound of $\delta_{P_i}$ in inequality (\ref{e21}) and the fact that the interiors of $P_i$ are disjoint for different $i$, there are at most a constant $C_n$ number of $i$ such that $g(m) \le i \le g(m+1)$, implying 
\begin{equation}\label{e23}
g(m+1) - g(m) \le C_n. 
\end{equation}

We let  $A = 2^{2C_n + 1} > 0$ and $a = (\frac{1}{2})^{\frac{1}{C_n}} \in (0, 1)$. We prove that the inequality in property (\ref{i1}) holds for these constants. For any $i, j \in \mathbb{N}$ satisfying $i \le j$, there exist $m, m' \in \mathbb{N}$ such that $g(m) \le i < g(m+1)$ and $g(m') \le j < g(m'+1)$. Then, by Lemma \ref{6}(\ref{i7}), equation (\ref{e22}), and inequality (\ref{e23}),
 \begin{align}\notag
\delta_{P_j} \le \delta_{P_{g(m')}} 2^{C_n} &\le \delta_{P_{g(m)}} \left(\frac{1}{2}\right)^{m' - m} \cdot 2^{C_n} \le 2^{2C_n} \delta_{P_i} a^{C_n (m' - m)}\\\notag
 &
\le 2^{2C_n + 1} \delta_{P_i} a^{g(m'+1) - g(m)} \le A a^{j - i} \delta_{P_i}.
\end{align}
This completes the proof of the property (\ref{i1}).  
\end{proof}

\begin{Remark}\label{19}
Let $\mathscr{F}$ be a decomposition outlined in Lemma \ref{1}.  Let $P\in\mathscr{F}$. The $A_P$ and $x_P$ are defined in Definition \ref{defn:paths} and Definition \ref{24}. For any $x\in A_P$, by the inequality (\ref{e24}) in the proof of Lemma \ref{2} and the Lemma \ref{1}(\ref{i5}), we have 
\begin{equation*}
|x - x_P|\lesssim\text{dist}(P, D)\lesssim\delta_P,
\end{equation*}
where the constant is an absolute constant.
\end{Remark}

\begin{Lemma}\label{12} Let $\mathscr{F}$ be a decomposition described in Lemma \ref{1}. For $P \in \mathscr{F}$, let $A_P$ and $x_P$ be defined as in Lemma \ref{2}, and let $x \in A_P$. Suppose $\{P_j\}_{j=0}^{\infty}$ is the family of cubes in $\mathscr{F}$ that satisfies the  properties (\ref{i9})-(\ref{i1}) in Lemma \ref{2}. Let $\{s\}>\frac{n}{p}$ and $\epsilon=\frac{\{s\}p - n}{2}$. Then, for any $F \in L^{s,p}(\mathbb{R}^n)$, the following inequality holds:
\begin{equation*}
\sum_{|i| \leq \lfloor s \rfloor} \delta_P^{n - sp + |i|p} \left| \partial^i J^{\lfloor s \rfloor}_{x_P} F(x) - \partial^i F(x) \right|^p \lesssim \delta_P^{n - \{s\}p + \epsilon} \sum_{j=1}^{\infty} \|F\|_{L^{s,p}(P_j)}^p \delta_{P_j}^{\{s\}p - n - \epsilon}.
\end{equation*}
\end{Lemma}

\begin{proof}
We only need to prove for each multi-index $i$, the corresponding item on the left hand side can be bounded by the right hand side:

For each $|i| \leq \lfloor s \rfloor$, we have $\partial^i F \in L^{s-|i|,p}(\mathbb{R}^n)$. By Proposition \ref{p1}, there exists $t \in (0, 1)$ such that
\begin{equation}\label{e3}
J_{x_P}^{\lfloor s \rfloor - |i|} \partial^i F(x) - \partial^i F(x) = \sum_{|k| = \lfloor s \rfloor - |i|} \frac{1}{k!} \left( \partial^{k+i} F(tx + (1-t)x_P) - \partial^{k+i} F(x_P) \right) (x - x_P)^k.
\end{equation}

There exists $m \in \mathbb{N}$ such that $tx + (1-t)x_P \in P_m$. Let $x_m := tx + (1-t)x_P$, and for any $j > m$, let $x_j \in P_{j-1} \cap P_j$. 

Using equation (\ref{e3}), Remark \ref{19}, H\"older's inequality, Lemma \ref{2}(\ref{i1}), and the Sobolev embedding theorem, we have:
\begin{align*}
\delta_P^{n-sp+|i|p} \left| \partial^i J^{\lfloor s \rfloor}_{x_P} F(x) - \partial^i F(x) \right|^p 
=\;& \delta_P^{n-sp+|i|p} \left| J^{\lfloor s \rfloor - |i|}_{x_P} \partial^i F(x) - \partial^i F(x) \right|^p \\
\lesssim \;& \delta_P^{n - \{s\}p} \sum_{|k| = \lfloor s \rfloor} \left| \partial^k F(x_m) - \partial^k F(x_P) \right|^p \\
\lesssim \; & \delta_P^{n - \{s\}p} \sum_{|k| = \lfloor s \rfloor} \bigg| \sum_{j=m}^{\infty} (\partial^k F(x_j) - \partial^k F(x_{j+1})) \bigg|^p \\
\lesssim & \sum_{|k| = \lfloor s \rfloor} \delta_P^{n - \{s\}p} \sum_{j=m}^{\infty} \left| \partial^k F(x_j) - \partial^k F(x_{j+1}) \right|^p \delta_{P_j}^{-\epsilon} \left( \sum_{j=m}^{\infty} \delta_{P_j}^{\frac{\epsilon}{p-1}} \right)^{p-1} \\
\lesssim \;& \delta_P^{n - \{s\}p + \epsilon} \sum_{j=1}^{\infty} \|F\|_{L^{s,p}(P_j)}^p \delta_{P_j}^{\{s\}p - n - \epsilon}
\end{align*}
 as claimed. This completes the proof of the lemma. 
\end{proof}

\begin{Lemma}\label{7}
Let $\mathscr{F}$ be a decomposition described in Lemma \ref{1}.  Then
\begin{enumerate}[(1)]
\item\label{i3} For any $Q,Q'\in\mathscr{F}$, if $Q\leftrightarrow Q'$, then
\begin{equation*}
\int_{Q'}\int_{Q}\frac{1}{|x-y|^{n+\{s\}p-p}}\mathrm{d}x\mathrm{d}y\lesssim\delta_Q^{n+p-\{s\}p}.
\end{equation*}
\item\label{i2} For any $ Q \in \mathscr{F}$,
\begin{equation*}
\sum_{Q':Q\not\leftrightarrow Q'}\int_{Q'}\int_{Q}\frac{1}{|x-y|^{n+\{s\}p}}\mathrm{d}x\mathrm{d}y\lesssim\delta_Q^{n-\{s\}p},
\end{equation*}
and
\begin{equation*}
\int_{D}\int_{Q}\frac{1}{|x-y|^{n+\{s\}p}}\mathrm{d}x\mathrm{d}y\lesssim\delta_Q^{n-\{s\}p}.
\end{equation*}
\end{enumerate}

\end{Lemma}

\begin{proof}
(1) 
Let  $z\in Q\bigcap Q'$. We know that such a $z$ exists  because $Q\leftrightarrow Q'$. For any $x\in Q$ and $y\in Q'$,      by Lemma \ref{6}(\ref{i7}), we have

\begin{equation*}
|x-y|\le|x-z|+|z-y|\le\delta_Q+\delta_{Q'}\le 3\delta_Q.
\end{equation*}
This implies that  $Q-y\subseteq  B_{3\delta_Q}(0)$. By changing of variables,
\begin{equation*}
\begin{split}
\int_{Q'}\int_{Q}\frac{1}{|x-y|^{n+\{s\}p-p}}\mathrm{d}x\mathrm{d}y\le&\int_{Q'}\int_{Q-y}\frac{1}{|x|^{n+\{s\}p-p}}\mathrm{d}x\mathrm{d}y\\
&\lesssim\int_{Q'}\int_{B_{3\delta_Q}(0)}\frac{1}{|x|^{n+\{s\}p-p}}\mathrm{d}x\mathrm{d}y\\
&\lesssim\delta_{Q'}^n\delta_Q^{p-\{s\}p}\int_{B_{1}(0)}\frac{1}{|x|^{n+\{s\}p-p}}\mathrm{d}x\\
&\lesssim\delta_Q^{n+p-\{s\}p},
\end{split}
\end{equation*}
where the last inequality follows that the integral 
\begin{equation*}
\int_{B_{1}(0)}\frac{1}{|x|^{n+\{s\}p-p}}\mathrm{d}x\mathrm{d}y
\end{equation*}
is finite and only depending on $n,s$ and $p$.

(2) By Lemma \ref{6}(\ref{i8}),  for any $Q'\not\leftrightarrow Q$, we have $Q'\bigcap (2Q)^\circ=\emptyset$. This implies that for any $x\in Q$,
\begin{equation}\label{e16}
Q'\subseteq(B_{\frac{1}{2\sqrt{n}}\delta_Q}(x))^c,
\end{equation}
and for any $x'\in D$,
by Lemma \ref{1}(\ref{i5}), $10\delta_Q\le$dist$(Q,D)\le|x'-x|$. Therefore 
\begin{equation}\label{e17}
D\subseteq(B_{10\delta_Q}(x))^c\subseteq(B_{\frac{1}{2\sqrt{n}}\delta_Q}(x))^c.
\end{equation}
By equation (\ref{e16}) and equation (\ref{e17}), using the fact that the interiors of different cubes are disjoint, we have 
\begin{align*}
\max &\left\{\int_{Q}\sum_{Q':Q\not\leftrightarrow Q'}\int_{Q'} \frac{1}{|x-y|^{n+\{s\}p}}\mathrm{d}x\mathrm{d}y, ~ \int_{Q} \int_{D}\frac{1}{|x-y|^{n+\{s\}p}}\mathrm{d}x\mathrm{d}y
\right\} \\
\le&\int_{Q}\int_{\big(B_{\frac{1}{2\sqrt{n}}\delta_Q}(x)\big)^c}\frac{1}{|x-y|^{n+\{s\}p}}\mathrm{d}y\mathrm{d}x\\
=&\int_{Q}\int_{\big(B_{\frac{1}{2\sqrt{n}}\delta_Q}(0)\big)^c}\frac{1}{|y|^{n+\{s\}p}}\mathrm{d}y\mathrm{d}x\\
\lesssim &\; \delta_Q^{n-\{s\}p}\int_{(B_{1}(0))^c}\frac{1}{|y|^{n+\{s\}p}}\mathrm{d}y\\
\lesssim &\; \delta_Q^{n-\{s\}p},
\end{align*}
where the last step follows from the fact that the integral in the fourth line is a finite constant only depending on $n,s$ and $p$.
\end{proof}
\section{Proof of main result Theorem \ref{main}}\label{proof}

If $\overline{E}$ is either empty or equal to $\mathbb{R}^n$, the theorem holds trivially. Therefore, we consider the case where $\overline{E}$ is neither of these extreme cases. 
 Then by Lemma \ref{1}, there exists a family of closed cubes $\mathscr{F}=\{Q_k\}$ such that 
\begin{enumerate}[(1)]
\item $\bigcup_{k}Q_k=(\overline{E})^c$.

\item The interiors of $Q_k$ are mutually disjoint.

\item 10$\delta_{Q_k}\le$ dist$(Q_k,\overline{E})\le 22\delta_{Q_k}$ for any $k$.

\item For every $k$, there exists $m\in\mathbb{Z}$ and $(a_i)_{i=1}^n\in\mathbb{Z}^n$ such that $Q_k=\prod_{i=1}^n[a_i2^{-m},(a_i+1)2^{-m}]$.
\end{enumerate}
There is a family of functions $\{\theta_P\}_{P\in \mathscr F}\subseteq C^\infty(\mathbb{R}^n)$ satisfying the conclusions in Lemma \ref{3} corresponding to the given family $\mathscr{F}$. By Definition \ref{defn:paths} and \ref{24}, we equip each $P\in\mathscr{F}$ with a set $A_P$ and a point $x_P\in \overline{E}$. Now we define the operator $T:J^{\lfloor s\rfloor}_EL^{s,p}(\mathbb{R}^n)\rightarrow L^{s,p}(\mathbb{R}^n)$ as follows:

Let $f\in J^{\lfloor s\rfloor}_EL^{s,p}(\mathbb{R}^n)$. By the definition of $J^{\lfloor s\rfloor}_EL^{s,p}(\mathbb{R}^n)$ in equation (\ref{15}), let $F\in L^{s,p}(\mathbb{R}^n)$ such that
\begin{equation}\label{e29}
(J^{\lfloor s\rfloor}_xF)_{x\in E}=f.
\end{equation}
Define 
\begin{equation*}
\begin{split}
Tf(x)=\begin{cases}
\sum_{P\in\mathscr{F}}\theta_P(x)J^{\lfloor s\rfloor}_{x_P}F(x),&x\in\mathbb{R}^n\backslash \overline{E};\\
F(x), &x\in \overline{E}.
\end{cases}
\end{split}
\end{equation*}

By the Sobolev embedding theorem, $F\in C^{\lfloor s\rfloor}(\mathbb{R}^n)$. By Lemma \ref{10}, the expression on the right hand side of the above equation is well defined for any given $F$. By equation (\ref{e29}) and continuity, the value $F(x)$ is independent of the choice of $F$ for any $x\in\overline{E}$. For any $P\in\mathscr{F}$, because $x_P\in\overline{E}$, $J_{x_P}^{\lfloor s\rfloor}F$ is independent of the choice of $F$ by continuity. Therefore, $T(f)$ is independent of the choice of $F$. So it is well-defined.

It is clear that  $T$ is a linear operator. We will prove it is an extension operator:

Let $f\in J_E^{\lfloor s \rfloor}L^{s,p}(\mathbb{R}^n)$. Let  $F\in L^{s,p}(\mathbb{R}^n)$   be any function satisfying $(J^{\lfloor s\rfloor}_xF)_{x\in E}=f$. By Lemma \ref{10}, $Tf\in C^{\lfloor s\rfloor}(\mathbb{R}^n)$ and for any $|i|\le\lfloor s\rfloor$ and $x\in \overline{E}$, we have 
\begin{equation}\label{e6}
\partial^i T(f)(x)=\partial^iF(x).
\end{equation}
Therefore
\begin{equation*}
(J^{\lfloor s\rfloor}_xT(f))_{x\in E}=(J^{\lfloor s\rfloor}_xF)_{x\in E}=f.
\end{equation*}
Therefore $T$ is an extension operator. Next we prove that the operator is bounded. 

For this, we only need to show that $\|Tf\|_{L^{s,p}(\mathbb{R}^n)} \lesssim \|F\|_{L^{s,p}(\mathbb{R}^n)}$.  
We will complete this task by evaluating integrals over four disjoint regions and by using the Whitney decomposition of $(\overline{E})^c$ that we constructed above, as detailed below:

By the definition of homogeneous fractional Sobolev space in equation (\ref{e14}), equation \eqref{e6}, Lemma \ref{9} and Lemma \ref{3}(\ref{i11}), we have
\begin{align*} 
&\|Tf\|_{L^{s,p}(\mathbb{R}^n)}^p\lesssim\sum_{|k|=\lfloor s\rfloor}\int_{\mathbb{R}^n}\int_{\mathbb{R}^n}\frac{|\partial^k(Tf)(x)-\partial^k(Tf)(y)|^p}{|x-y|^{n+\{s\}  p}}\mathrm{d}x\mathrm{d}y\\
 \lesssim& \sum_{|k|=\lfloor s\rfloor}\int_{\overline{E}}\int_{\overline{E}}\frac{|\partial^kF(x)-\partial^kF(y)|^p}{|x-y|^{n+\{s\}p}}\mathrm{d}x\mathrm{d}y\\
&+\sum_{|k|=\lfloor s\rfloor}\sum_{Q,Q':Q\not\leftrightarrow Q'}\int_{Q'}\int_{Q}\frac{|\sum_{P\in \mathscr{F} }\partial^k(\theta_P(x)J^{\lfloor s\rfloor}_{x_P}F(x))-\sum_{P\in \mathscr{F}}\partial^k(\theta_P(y)J^{\lfloor s\rfloor}_{x_P}F(y))|^p}{|x-y|^{n+\{s\}p}}\mathrm{d}x\mathrm{d}y\\
&+\sum_{|k|=\lfloor s\rfloor}\sum_{Q,Q':Q\leftrightarrow Q'}\int_{Q'}\int_{Q}\frac{|\sum_{P\in \mathscr{F}}\partial^k(\theta_P(x)J^{\lfloor s\rfloor}_{x_P}F(x))-\sum_{P\in \mathscr{F}}\partial^k(\theta_P(y)J^{\lfloor s\rfloor}_{x_P}F(y))|^p}{|x-y|^{n+\{s\}p}}\mathrm{d}x\mathrm{d}y\\
&+\sum_{|k|=\lfloor s\rfloor}\sum_Q\int_{\overline{E}}\int_{Q}\frac{|\sum_{P\in \mathscr{F}}\partial^k(\theta_P(x)J^{\lfloor s\rfloor}_{x_P}F(x))-\partial^kF(y)|^p}{|x-y|^{n+\{s\}p}}\mathrm{d}x\mathrm{d}y\\
=& \sum_{|k|=\lfloor s\rfloor}
\left(\text{I}_k + \text{II}_k+ \text{III}_k +\text{IV}_k\right). 
\end{align*}

\vskip.1in 

Notice that by the definition of $\|F\|_{L^{s,p}(\mathbb{R}^n)}$, we already  have 
\begin{equation}\label{20}
\text{I}_k\le\|F\|_{L^{s,p}(\mathbb{R}^n)}^p.
\end{equation}

\vskip.1in 

For the rest, we will first compute $\text{II}_k$. Using Lemma \ref{3}(\ref{i4}), we can express $\text{II}_k$ as the sum of three components as follows: 
\begin{align}\notag
\text{II}_k =&\sum_{Q,Q':Q\not\leftrightarrow Q'}\int_{Q'}\int_{Q}\frac{|\sum_P\partial^k(\theta_P(x)J^{\lfloor s\rfloor}_{x_P}F(x))-\sum_P\partial^k(\theta_P(y)J^{\lfloor s\rfloor}_{x_P}F(y))|^p}{|x-y|^{n+\{s\}p}}\mathrm{d}x\mathrm{d}y\\\label{eq:1}
\lesssim&\sum_{Q,Q':Q\not\leftrightarrow Q'}\int_{Q'}\int_{Q}\frac{|\sum_{0\le i\le k}\sum_P\partial^{k-i}\theta_P(x)(\partial^iJ^{\lfloor s\rfloor}_{x_P}F(x)-\partial^iF(x))|^p}{|x-y|^{n+\{s\}p}}\mathrm{d}x\mathrm{d}y\\\label{eq:2}
&+\sum_{Q,Q':Q\not\leftrightarrow Q'}\int_{Q'}\int_{Q}\frac{|\partial^kF(x)-\partial^kF(y)|^p}{|x-y|^{n+\{s\}p}}\mathrm{d}x\mathrm{d}y\\\label{eq:3}
&+\sum_{Q,Q':Q\not\leftrightarrow Q'}\int_{Q'}\int_{Q}\frac{|\sum_{0\le i\le k}\sum_P\partial^{k-i}\theta_P(y)(\partial^iJ^{\lfloor s\rfloor}_{x_P}F(y)-\partial^iF(y))|^p}{|x-y|^{n+\{s\}p}}\mathrm{d}x\mathrm{d}y.
\end{align}
Observe that the  terms \eqref{eq:1} and \eqref{eq:3} in the sum above  are identical by swapping the variables $x$ and $y$. Now  
we will bound the term  \eqref{eq:1}. By applying Lemma \ref{6}(\ref{i8}), Lemma \ref{3}(\ref{i11}) and Remark \ref{11}, we obtain:  
\begin{equation}\eqref{eq:1} \lesssim 
 \sum_{Q,Q':Q\not\leftrightarrow Q'}\sum_{P:P\leftrightarrow Q} \sup_{t\in Q}\sum_{0\le i\le k} \delta_P^{(|i|-|k|)p} |\partial^iJ^{\lfloor s\rfloor}_{x_P}F(t)-\partial^iF(t)|^p\int_{Q'}\int_{Q}\frac{1}{|x-y|^{n+\{s\}p}}\mathrm{d}x\mathrm{d}y. \label{e15}
\end{equation}
Note that by Lemma \ref{14} and Lemma \ref{7}(\ref{i2}),
\begin{align*}
(\ref{e15})\lesssim &\sum_P\sup_{t\in \bigcup_{Q\leftrightarrow P}Q}\sum_{0\le i\le k}\delta_P^{|i|p+n-sp}|\partial^iJ^{\lfloor s\rfloor}_{x_P}F(t)-\partial^iF(t)|^p\\
\le&\sum_P\sup_{t\in A_P}\sum_{|i|\le\lfloor s\rfloor}\delta_P^{|i|p+n-sp}|\partial^iJ^{\lfloor s\rfloor}_{x_P}F(t)-\partial^iF(t)|^p.
\end{align*}
Combining the above estimations on \eqref{eq:1} and \eqref{eq:3} (by symmetry), and using the upper bound $\|F\|_{L^{s,p}(\mathbb{R}^n)}^p$ for the middle term \eqref{eq:2}, we obtain: 
\begin{equation}\label{21}
\text{II}_k\lesssim\sum_P\sup_{t\in A_P}\sum_{|i|\le\lfloor s\rfloor}\delta_P^{|i|p+n-sp}|\partial^iJ^{\lfloor s\rfloor}_{x_P}F(t)-\partial^iF(t)|^p+\|F\|_{L^{s,p}(\mathbb{R}^n)}^p.
\end{equation}

\vskip.1in 

Next, we need to compute $\text{III}_k$. 
\begin{equation*}
\begin{split}
\text{III}_k=&\sum_{Q,Q':Q\leftrightarrow Q'}\int_{Q'}\int_{Q}\frac{|\sum_P\partial^k(\theta_P(x)J^{\lfloor s\rfloor}_{x_P}F(x))-\sum_P\partial^k(\theta_P(y)J^{\lfloor s\rfloor}_{x_P}F(y))|^p}{|x-y|^{n+\{s\}p}}\mathrm{d}x\mathrm{d}y\\
\lesssim&\sum_{Q,Q':Q\leftrightarrow Q'}\int_{Q'}\int_{Q}\frac{|\sum_P\sum_{0\le i\le k}(\partial^{k-i}\theta_P(x)-\partial^{k-i}\theta_P(y))\partial^iJ^{\lfloor s\rfloor}_{x_P}F(x)|^p}{|x-y|^{n+\{s\}p}}\mathrm{d}x\mathrm{d}y\\
&+\sum_{Q,Q':Q\leftrightarrow Q'}\int_{Q'}\int_{Q}\frac{|\sum_P\sum_{0\le i\le k}\partial^{k-i}\theta_P(y)(\partial^iJ^{\lfloor s\rfloor}_{x_P}F(x)-\partial^i J^{\lfloor s\rfloor}_{x_P}F(y))|^p}{|x-y|^{n+\{s\}p}}\mathrm{d}x\mathrm{d}y\\
=&\text{III}_{k,1}+
\text{III}_{k,2}.
\end{split}
\end{equation*}

First we compute III$_{k,1}$ as follows. Using Lemma \ref{3}(\ref{i4}), observe that $\sum_P \partial^i(\theta_P(x) - \theta_P(y)) = 0$ for any multi-index $i$, and $x,y \in (\overline{E})^c$. Therefore, 
\begin{align}\notag
\text{III}_{k,1}=&\sum_{Q,Q':Q\leftrightarrow Q'}\int_{Q'}\int_{Q}\frac{|\sum_P\sum_{0\le i\le k}(\partial^{k-i}\theta_P(x)-\partial^{k-i}\theta_P(y))\partial^iJ^{\lfloor s\rfloor}_{x_P}F(x)|^p}{|x-y|^{n+\{s\}p}}\mathrm{d}x\mathrm{d}y\\\notag
=&\sum_{Q,Q':Q\leftrightarrow Q'}\int_{Q'}\int_{Q}\frac{|\sum_P\sum_{0\le i\le k}(\partial^{k-i}\theta_P(x)-\partial^{k-i}\theta_P(y))(\partial^iJ^{\lfloor s\rfloor}_{x_P}F(x)-\partial^i F(x))|^p}{|x-y|^{n+\{s\}p}}\mathrm{d}x\mathrm{d}y.
\end{align}
Applying Lemma \ref{6}(\ref{i8}) and the mean value theorem, we continue the estimation:
\begin{align}\notag
\text{III}_{k,1} \le&\sum_{Q,Q':Q\leftrightarrow Q'}\int_{Q'}\int_{Q}\frac{\left(\sum_{P:P\leftrightarrow Q \text{ or }P\leftrightarrow Q'}\sum_{0\le i\le k}|\sup_{t\in[x,y]}\nabla\partial^{k-i}\theta_P(t)||\partial^iJ^{\lfloor s\rfloor}_{x_P}F(x)-\partial^i F(x)|\right)^p}{|x-y|^{n+\{s\}p-p}}\mathrm{d}x\mathrm{d}y\\\notag
\lesssim&\sum_{Q,Q':Q\leftrightarrow Q'}\sum_{P:P\leftrightarrow Q \text{ or }P\leftrightarrow Q'}\sup_{t\in Q}\sum_{0\le i\le k}\delta_P^{-(|k|-|i|+1)p}|\partial^iJ^{\lfloor s\rfloor}_{x_P}F(t)-\partial^i F(t)|^p\int_{Q'}\int_{Q}\frac{1}{|x-y|^{n+\{s\}p-p}}\mathrm{d}x\mathrm{d}y.
\end{align}
We obtained the preceding line 
 by Remark \ref{11}. Now, by an application of  Lemma \ref{7}(\ref{i3}) and Lemma \ref{14} we obtain the following: 
\begin{align}\notag
\text{III}_{k,1}\lesssim & \sum_{Q,Q':Q\leftrightarrow Q'}\sum_{P:P\leftrightarrow Q \text{ or }P\leftrightarrow Q'}\sup_{t\in A_P}\sum_{0\le i\le k}\delta_P^{n-sp+|i|p}|\partial^iJ^{\lfloor s\rfloor}_{x_P}F(t)-\partial^i F(t)|^p \\\label{e12}
\lesssim&\sum_P\sup_{t\in A_P}\sum_{|i|\le\lfloor s\rfloor}\delta_P^{n-sp+|i|p}|\partial^iJ^{\lfloor s\rfloor}_{x_P}F(t)-\partial^i F(t)|^p. 
\end{align}
This completes the estimation of $\text{III}_{k,1}$. 
\medskip

Next, we estimate 
$\text{III}_{k,2}$.  Notice that 
for any $x,y\in\mathbb{R}^n$ and $P\in\mathscr{F}$, we have 

\begin{equation}\label{e8}
\partial^kJ^{\lfloor s\rfloor}_{x_P}F(x)=\partial^kF(x_P)=\partial^k J^{\lfloor s\rfloor}_{x_P}F(y).
\end{equation}
Therefore, by equation (\ref{e8}), we omit the term for $i = k$ in $\text{III}_{k,2}$ and write:

\begin{equation}\notag
\text{III}_{k,2}=\sum_{Q,Q':Q\leftrightarrow Q'}\int_{Q'}\int_{Q}\frac{|\sum_P\sum_{0\le i< k}\partial^{k-i}\theta_P(y)(\partial^iJ^{\lfloor s\rfloor}_{x_P}F(x)-\partial^i J^{\lfloor s\rfloor}_{x_P}F(y))|^p}{|x-y|^{n+\{s\}p}}\mathrm{d}x\mathrm{d}y.
\end{equation}
Using Lemma \ref{3}(\ref{i4}), Lemma \ref{6}(\ref{i8}), and the mean value theorem once again, we get: 

\begin{align}\notag
\text{III}_{k,2}=&\sum_{Q,Q':Q\leftrightarrow Q'}\int_{Q'}\int_{Q}\frac{\big|\sum_P\sum_{0\le i<k}\partial^{k-i}\theta_P(y)\big((\partial^iJ^{\lfloor s\rfloor}_{x_P}F(x)-\partial^iF(x))-(\partial^i J^{\lfloor s\rfloor}_{x_P}F(y)-\partial^iF(y))\big)\big|^p}{|x-y|^{n+\{s\}p}}\mathrm{d}x\mathrm{d}y\\\notag
\le&\sum_{Q,Q':Q\leftrightarrow Q'}\int_{Q'}\int_{Q}\frac{(\sum_{P\leftrightarrow Q'}\sum_{0\le i< k}\sup_{t\in[x,y]}|\partial^{k-i}\theta_P(y)\nabla(\partial^iJ^{\lfloor s\rfloor}_{x_P}F-\partial^iF)(t)\cdot(x-y)|)^p}{|x-y|^{n+\{s\}p}}\mathrm{d}x\mathrm{d}y. 
\end{align}
By Remark (\ref{11}) and a change of  variable for $i$, we have

\begin{align*}
\text{III}_{k,2}\lesssim&\sum_{Q,Q':Q\leftrightarrow Q'}\int_{Q'}\int_{Q}\frac{\sum_{P\leftrightarrow Q'}\sum_{0\le i<k}\sup_{t\in[x,y]}\delta_P^{|i|p-\lfloor s\rfloor p}|\nabla\partial^iJ^{\lfloor s\rfloor}_{x_P}F(t)-\nabla\partial^iF(t)|^p}{|x-y|^{n+\{s\}p-p}}\mathrm{d}x\mathrm{d}y\\
\lesssim&\sum_{Q,Q':Q\leftrightarrow Q'}\int_{Q'}\int_{Q}\frac{\sum_{P\leftrightarrow Q'}\sum_{|i|\le\lfloor s\rfloor}\sup_{t\in[x,y]}\delta_P^{|i|p-p-\lfloor s\rfloor p}|\partial^iJ^{\lfloor s\rfloor}_{x_P}F(t)-\partial^iF(t)|^p}{|x-y|^{n+\{s\}p-p}}\mathrm{d}x\mathrm{d}y.\addtocounter{equation}{1}\tag{\theequation}\label{e11}\\
\end{align*}
By Lemma \ref{7}(\ref{i3}) and Lemma \ref{14},
\begin{align*}
(\ref{e11})\le&\sum_{Q,Q':Q\leftrightarrow Q'}\sum_{P\leftrightarrow Q'}\sum_{|i|\le\lfloor s\rfloor}\sup_{t\in A_P}\delta_P^{|i|p-p-\lfloor s\rfloor p}|\partial^iJ^{\lfloor s\rfloor}_{x_P}F(t)-\partial^iF(t)|^p\int_{Q'}\int_{Q}\frac{1}{|x-y|^{n+\{s\}p-p}}\mathrm{d}x\mathrm{d}y\\
\lesssim&\sum_P\sup_{t\in A_P}\sum_{|i|\le\lfloor s\rfloor}\delta_P^{n-sp+|i|p}|\partial^iJ^{\lfloor s\rfloor}_{x_P}F(t)-\partial^i F(t)|^p.
\end{align*}
A combination of the estimations above yields 

\begin{equation}\label{e13}
\text{III}_{k,2}\lesssim\sum_P\sup_{t\in A_P}\sum_{|i|\le\lfloor s\rfloor}\delta_P^{n-sp+|i|p}|\partial^iJ^{\lfloor s\rfloor}_{x_P}F(t)-\partial^i F(t)|^p.
\end{equation}
Combining inequality (\ref{e12}) and (\ref{e13}), we get

\begin{equation}\label{22}
\text{III}_k=\text{III}_{k,1}+\text{III}_{k,2}\lesssim\sum_P\sup_{t\in A_P}\sum_{|i|\le\lfloor s\rfloor}\delta_P^{n-sp+|i|p}|\partial^iJ^{\lfloor s\rfloor}_{x_P}F(t)-\partial^i F(t)|^p.
\end{equation}

Next we compute 
$\text{IV}_k$. As before, we use the variant of triangle inequality, to bound:
\begin{align*}
\text{IV}_k=&\sum_Q\int_{\overline{E}}\int_{Q}\frac{|\sum_P\partial^k(\theta_P(x)J^{\lfloor s\rfloor}_{x_P}F(x))-\partial^kF(y)|^p}{|x-y|^{n+\{s\}p}}\mathrm{d}x\mathrm{d}y\\
\lesssim&\sum_Q\int_{\overline{E}}\int_{Q}\frac{|\sum_{0\le i\le k}\sum_{P\leftrightarrow Q}\partial^{k-i}\theta_P(x)(\partial^iJ^{\lfloor s\rfloor}_{x_P}F(x)-\partial^iF(x))|^p}{|x-y|^{n+\{s\}p}}\mathrm{d}x\mathrm{d}y\\
&+\int_{\overline{E}}\int_{Q}\frac{|\partial^kF(x)-\partial^kF(y)|^p}{|x-y|^{n+\{s\}p}}\mathrm{d}x\mathrm{d}y\hskip11.5em\text{(by Lemma \ref{3} (\ref{i4}))}\\
\lesssim&\sum_Q\sum_{P\leftrightarrow Q} \sup_{t\in Q}\sum_{0\le i\le k}\delta_P^{(|i|-|k|)p} |\partial^iJ^{\lfloor s\rfloor}_{x_P}F(t)-\partial^i F(t)|^p\int_{\overline{E}}\int_{Q}\frac{1}{|x-y|^{n+\{s\}p}}\mathrm{d}x\mathrm{d}y \\
& + \quad \|F\|_{L^{s,p}(\mathbb{R}^n)}^p\hspace{8cm} \text{(by Remark \ref{11})}\\
\lesssim&\sum_P\sup_{t\in A_P}\sum_{|i|\le\lfloor s\rfloor}\delta_P^{n-sp+|i|p}|\partial^iJ^{\lfloor s\rfloor}_{x_P}F(t)-\partial^i F(t)|^p+\|F\|_{L^{s,p}(\mathbb{R}^n)}^p\quad \text{(by Lemma \ref{7}(\ref{i2}), Lemma \ref{14})}.\addtocounter{equation}{1}\tag{\theequation}\label{23}
\end{align*}

Combining the results from the computations of terms \hyperref[20]{I$_k$}, \hyperref[21]{II$_k$}, \hyperref[22]{III$_k$}, and  \hyperref[23]{IV$_k$} above, we obtain 
\begin{align*}
\|Tf\|_{L^{s,p}(\mathbb{R}^n)}^p\lesssim&\sum_{k:~ |k|=\lfloor s\rfloor}\Big(\sum_{P}\sup_{t\in A_P}\sum_{|i|\le\lfloor s\rfloor}\delta_P^{n-sp+|i|p}|\partial^iJ^{\lfloor s\rfloor}_{x_P}F(t)-\partial^i F(t)|^p+\|F\|_{L^{s,p}(\mathbb{R}^n)}^p\Big)\\
\lesssim&\sum_P\sup_{t\in A_P}\sum_{|i|\le\lfloor s\rfloor}\delta_P^{n-sp+|i|p}|\partial^iJ^{\lfloor s\rfloor}_{x_P}F(t)-\partial^i F(t)|^p+\|F\|_{L^{s,p}(\mathbb{R}^n)}^p.
\end{align*}

Let $\epsilon=\frac{\{s\}p-n}{2}$. For any $P\in\mathscr{F}$, let $a_P\in A_P$ such that
\begin{equation*}
\sum_{|i|\le\lfloor s\rfloor}\delta_P^{n-sp+|i|p}|\partial^iJ^{\lfloor s\rfloor}_{x_P}F(a_P)-\partial^i F(a_P)|^p=\sup_{t\in A_P}\sum_{|i|\le\lfloor s\rfloor}\delta_P^{n-sp+|i|p}|\partial^iJ^{\lfloor s\rfloor}_{x_P}F(t)-\partial^i F(t)|^p.
\end{equation*}

For each $P \in \mathscr{F}$ let $(P_j=P_j(P))_{j=0}^{\infty}$ be the sequence of cubes generated in Lemma \ref{2} corresponding to cube $P$ and the point $x=a_P$.
Then by Lemma \ref{12},
\begin{equation}\label{13}
\begin{split}
\|Tf\|_{L^{s,p}(\mathbb{R}^n)}^p\lesssim&\sum_P\sum_{|i|\le\lfloor s\rfloor}\delta_P^{n-sp+|i|p}|\partial^iJ^{\lfloor s\rfloor}_{x_P}F(a_P)-\partial^i F(a_P)|^p+\|F\|_{L^{s,p}(\mathbb{R}^n)}^p\\
\lesssim&\sum_P\delta_P^{n-\{s\}p+\epsilon}\sum_{j=0}^{\infty}\|F\|_{L^{s,p}(P_j(P))}^p\delta_{P_j(P)}^{\{s\}p-n-\epsilon}+\|F\|_{L^{s,p}(\mathbb{R}^n)}^p.
\end{split}
\end{equation}
For two cubes $P$ and $P'$ in $\mathscr{F}$, we say $P \rightarrow P'$ if there exists a $j$ such that $P' = P_j(P)$.
With this, we change the order of the summation in the previous sum to get:
\begin{align}\label{innersum}
\sum_P\delta_{P}^{n-\{s\}p+\epsilon}\sum_{j=0}^{\infty}\|F\|_{L^{s,p}(P_j(P))}^p\delta_{P_j(P)}^{\{s\}p-n-\epsilon}=\sum_{P'}\|F\|_{L^{s,p}(P')}^p\left(\delta_{P'}^{\{s\}p-n-\epsilon}\sum_{P:P\rightarrow P'}\delta_{P}^{n-\{s\}p+\epsilon}\right).
\end{align}
Recall that for all $i<j$, we have $\delta_{P_j}\le Aa^{j-i}\delta_{P_i}$ (see item (\ref{i1}) in Lemma \ref{2}). By this and item (\ref{i10}) in Lemma \ref{2}, for each cube $P'$, if $P \rightarrow P'$, then 
\begin{equation*}
\text{dist}(P,P') \leq \sum_{j=0}^{\infty} \delta_{P_j} \lesssim \delta_{P},
\end{equation*}
where we have written $P_j = P_j(P)$. So for each $\ell\in\mathbb{Z}$, there are at most $C$ cubes $P$ such that $\delta_P=2^\ell$ and $P\rightarrow P'$. Then, for the inner sum on the far right of \eqref{innersum}, using that $n - \{s\}p + \epsilon < 0$, we have
\begin{equation*}
\begin{split}
\sum_{P:P\rightarrow P'}\delta_{P}^{n-\{s\}p+\epsilon}=&\sum_{P:P\rightarrow P' \text{ and }\delta_P\ge\frac{\delta_{P'}}{A}}\delta_{P}^{n-\{s\}p+\epsilon}\\
=&\sum_{\ell\ge\log_2\frac{\delta_{P'}}{A}}\,\,\,\sum_{P:P\rightarrow P' \text{ and }\delta_P=2^\ell}\delta_{P}^{n-\{s\}p+\epsilon}\\
\lesssim&\sum_{\ell\ge\log_2\frac{\delta_{P'}}{A}}(2^{n-\{s\}p+\epsilon})^\ell\\
\lesssim&\;\delta_{P'}^{n-\{s\}p+\epsilon}.
\end{split}
\end{equation*}
Using this in \eqref{innersum},  we obtain 
\begin{equation*}
\sum_P\delta_{P}^{n-\{s\}p+\epsilon}\sum_{j=0}^{\infty}\|F\|_{L^{s,p}(P_j(P))}^p\delta_{P_j(P)}^{\{s\}p-n-\epsilon}\lesssim\sum_{P'}\|F\|_{L^{s,p}(P')}^p\leq \|F\|_{L^{s,p}(\mathbb{R}^n)}^p. 
 \end{equation*}
Combining this with equation (\ref{13}), we obtain that 
\begin{equation*}
\|Tf\|_{L^{s,p}(\mathbb{R}^n)}^p\lesssim\|F\|_{L^{s,p}(\mathbb{R}^n)}^p.
\end{equation*}
This shows the boundedness of $T$ and completes the proof of our main result, Theorem \ref{main}. 

\begin{appendices}
\section{Proof of Lemma \ref{10}}\label{A1}
\begin{proof}
First, by Lemma \ref{9} and Lemma \ref{3}(\ref{i11}), $G$ is well defined and $G$ is $C^{\infty}$ on $D^c$ because the sum in $G$ is a finite sum in a neighborhood of any point in $D^c$.
Let $H=G-F$. Because $F\in C^m(\mathbb{R}^n)$, we get $H$ is $C^m$ on $D^c$. By Lemma \ref{3}(\ref{i4}), we have
\begin{equation*}
H(x)=\begin{cases}
\sum_{P\in\mathscr{F}}\theta_P(x)(J_{x_P}^{m}F(x)-F(x)),&x\in\mathbb{R}^n\backslash D;\\
0,&x\in D.
\end{cases}
\end{equation*}

For any $k$ such that $0 \le k \le m$, let $S(k)$ be the statement that for any multi-index $i$ with $0 \le |i| < k$, the partial derivative $\partial^i H$ is differentiable, and for any $0 \le |i| \le k$, we have $\partial^i H(x) = 0$ for all $x \in D$. It is clear that $S(0)$ is true.

For any  multi-index $i$ with  $|i|\le m$, for any $x_0\in D$, and  any $x\in D^c$, there exists $Q\in\mathscr{F}$ such that $x\in Q$. For any cube  $P$ such that $P\leftrightarrow Q$,  we have the following results   by applying Lemmas \ref{1}(\ref{i5}) and \ref{6}(\ref{i7}):
\begin{equation*}
|x-x_P|\le\delta_Q+\delta_P+\text{dist}(P,D)\le 
\delta_Q+23\delta_P  \le47\delta_Q,
\end{equation*}
and
\begin{equation}\label{e26}
\delta_Q\le\frac{1}{10}\text{dist}(Q,D)\le\frac{1}{10}|x-x_0|.
\end{equation}
Combining the above two inequalities together, we get
\begin{equation}\label{e4}
|x-x_P|\le\frac{47}{10}|x-x_0|.
\end{equation}
From inequality (\ref{e26}) and Lemma \ref{6}(\ref{i7}), we get
\begin{equation}\label{e25}
\delta_P\le2\delta_Q\le\frac{1}{5}|x-x_0|.
\end{equation}

Next, we will derive an upper bound for $\partial^i H(x)$:

\begin{align*}
|\partial^i H(x)|=&\Big|\partial^i\sum_P\theta_P(x)\big(J_{x_P}^{m}F(x)-F(x)\big)\Big|\\
\lesssim&\sum_{j\le i}\sum_P\big|\partial^{i-j}\theta_P(x)\big|\big|\partial^jJ_{x_P}^{m}F(x)-\partial^jF(x)\big|\\
\lesssim&\sum_{j\le i}\sum_{P\leftrightarrow Q}\delta_P^{|j|-|i|}
\big|J_{x_P}^{m-|j|}\partial^jF(x)-\partial^jF(x)\big| &(\text{by Remark \ref{11})}\\
\lesssim&\sum_{j\le i}\sum_{P\leftrightarrow Q}\delta_P^{|j|-|i|}\sum_{|h|=m}\sup_{|t-x|\le|x-x_P|}|\partial^hF(t)-\partial^hF(x_P)||x-x_P|^{m-|j|} &(\text{by Proposition \ref{p1}})\\\label{by(6)}
\lesssim&\sum_{P\leftrightarrow Q}\sum_{|h|=m}\sup_{|t-x|\le 5|x-x_0|}|\partial^hF(t)-\partial^hF(x_P)||x-x_0|^{m-|i|}\addtocounter{equation}{1}\tag{\theequation},
\end{align*}
where the latter  inequality is derived from    inequality  \eqref{e4} and inequality (\ref{e25}), all constants are only depending on $n$ and $m$. By triangle inequality and inequality (\ref{e4}), we shall continue as follows:
\begin{align*}
\eqref{by(6)} ~ \lesssim &\sum_{P\leftrightarrow Q}\sum_{|h|=m}\sup_{|t-x|\le 5|x-x_0|}|\partial^hF(t)-\partial^hF(x_0)||x-x_0|^{m-|i|}  \\
&+\sum_{P\leftrightarrow Q}\sum_{|h|=m}\sup_{|t-x|\le 5|x-x_0|}|\partial^hF(x_0)-\partial^hF(x_P)||x-x_0|^{m-|i|}\\
\lesssim&|x-x_0|^{m-|i|}\sum_{|h|=m}\sup_{|t-x|\le 5|x-x_0|}|\partial^hF(t)-\partial^hF(x_0)|.\addtocounter{equation}{1}\tag{\theequation}\label{e5}
\end{align*}

For any $0 \le k < m$, assume that the statement $S(k)$ is true. Then, for any $i$ with $|i| = k$, and for any $x_0 \in D$ and $x \in D^c$, we have $\partial^i H(x_0) = 0$ because $S(k)$ is true. Therefore, by inequality (\ref{e5}), we can write:
\begin{equation*}
\begin{split}
 |\partial^iH(x)-\partial^iH(x_0)| =& |\partial^iH(x)| \\
\lesssim&|x-x_0|^{m-|i|}\sum_{|h|=m}\sup_{|t-x|\le 5|x-x_0|}|\partial^hF(t)-\partial^hF(x_0)|.
\end{split}
\end{equation*}
Notice that 
\begin{equation*}
|x-x_0|^{m-1-|i|}\sum_{|h|=m}\sup_{|t-x|\le 5|x-x_0|}|\partial^hF(t)-\partial^hF(x_0)| 
\rightarrow 0 \quad \text{as}   \quad x\rightarrow x_0, \ x\in D^c,
\end{equation*}
because $|i|=k<m$.
Therefore
\begin{equation*}
\frac{|\partial^iH(x)-\partial^iH(x_0)|}{|x-x_0|}\rightarrow 0 \quad \text{as}   \quad x\rightarrow x_0, \ x\in D^c.
\end{equation*}
Also, for any $x\in D$, $\partial^iH(x)=\partial^iH(x_0)=0$. In conclusion,  we have    
\begin{equation}\label{e27}
\frac{|\partial^iH(x)-\partial^iH(x_0)|}{|x-x_0|} \rightarrow 0 \quad \text{as} \quad x\rightarrow x_0.
\end{equation}
This implies that $\partial^i H$ is differentiable at $x_0\in D$. Recall $H$ is $C^m$ on $D^c$. Therefore, $\partial^i H$ is differentiable on the entire $\mathbb{R}^n$. Furthermore, by the formula (\ref{e27}) and the definition of derivative,
for any $j$ such that $|j|=|i|+1=k+1$, we have  $\partial^j H(x)=0$. 

By induction, we conclude that $S(k)$ is true for all $0 \le k \le m$, and therefore also for $k = m$. This means that $H$ is $m$-times differentiable and that for any $|i| \le m$ and $x \in D$, 
\begin{equation}\label{e28}
\partial^i H(x) = 0.
\end{equation}

The final step is to prove that $\partial^i G$ is continuous for $|i| = m$. We only need to show that it is continuous at every point in $D$.

For any $x_0\in D$ and    any $x\in D^c$, using  inequality (\ref{e5}), we have 
\begin{equation*}
|\partial^iH(x)-\partial^iH(x_0)|=|\partial^iH(x)|\lesssim\sum_{|h|=m}\sup_{|t-x|\le 5|x-x_0|}|\partial^hF(t)-\partial^hF(x_0)| \rightarrow 0 \quad \text{as} \quad x\rightarrow x_0,\,x\in D^c.
\end{equation*}
For any $x\in D$, by equation (\ref{e28}),
\begin{equation*}
|\partial^iH(x)-\partial^iH(x_0)|=0.
\end{equation*}
Therefore 
\begin{equation*}
\lim_{x\rightarrow x_0}|\partial^iH(x)-\partial^iH(x_0)|=0.
\end{equation*}
This implies that $\partial^i H$ is continuous at every point in $D$, which in turn means that $\partial^i H$ is continuous on $\mathbb{R}^n$.
Since $H \in C^m(\mathbb{R}^n)$, it follows that $G = H + F \in C^m(\mathbb{R}^n)$. For any $x \in D$ and for all multi-indices $i$ with $|i| \leq m$, we have 
\begin{equation*}
\partial^i G(x) = \partial^i H(x) + \partial^i F(x) = \partial^i F(x).
\end{equation*}
\end{proof}
\end{appendices}

\printbibliography 

@article{Chang2017The,
title = "The Whitney extension theorem in high dimensions",
author = "Alan Chang",
year = "2017",
volume = "33",
pages = "623--632",
journal = "Revista Matematica Iberoamericana",
number = "2",
}

@article{Fefferman2006Whitney,
 author = {Fefferman, Charles L.},
 journal = {Annals of Mathematics},
 number = {1},
 pages = {313--359},
 publisher = {Annals of Mathematics},
 title = {Whitney's Extension Problem for ${C}^{m}$},
 volume = {164},
 year = {2006}
}

@article{Fefferman2014Sobolev,
 author = {Fefferman, Charles L. and Arie Israel and Garving K. Luli},
 journal = {Journal of the American Mathematical Society},
 number = {1},
 pages = {69--145},
 publisher = {American Mathematical Society},
 title = {Sobolev extension by linear operators},
 volume = {27},
 year = {2014}
}

@article{Fefferman2016Finiteness,
  title={Finiteness Principles for Smooth Selection},
  author={Fefferman, Charles L. and Israel, Arie and Luli, Garving K.},
  journal={Geometric and Functional Analysis},
  volume={26},
  pages={422--477},
  year={2016}
}

@article{Fefferman2017Interpolation,
  title={Interpolation of data by smooth nonnegative functions},
  author={Fefferman, Charles L. and Israel, Arie and Luli, Garving K.},
  journal={Revista matem{\'a}tica iberoamericana},
  volume={33},
  number={1},
  pages={305--324},
  year={2017}
}

@article{Fefferman2021C2,
  title={${C}^2$ interpolation with range restriction},
  author={Fefferman, Charles L. and Jiang, Fushuai and Luli, Garving K.},
  journal={Revista matem{\'a}tica iberoamericana},
  volume={39},
  number={2},
  pages={649--710},
  year={2023}
}

@article{Frerick2016Whitney,
 author = {Leonhard Frerick and Enrique Jord\'a and Jochen Wengenroth},
 journal = {Rev. Mat. Iberoam.},
 number = {2},
 pages = {377--390},
 title = {Whitney extension operators without loss of derivatives},
 volume = {32},
 year = {2016}
}

@article{Glaeser1958Etude,
  title={{\'E}tude de Quelques Alg{\`e}bres Tayloriennes},
  author={Georges Glaeser},
  journal={Journal d’Analyse Math{\'e}matique},
  year={1958},
  volume={6},
  pages={1-124},
}

@article{Israel2013A,
 author = {Arie Israel},
 journal = {Annals of Mathematics},
 number = {1},
 pages = {183--230},
 publisher = {Annals of Mathematics},
 title = {A bounded linear extension operator for ${L}^{2,p}(\mathbb{R}^2)$},
 volume = {178},
 year = {2013}
}

@book{Leoni2023A,
 author = {Giovanni Leoni},
 publisher = {American Mathematical Society},
 title = {A First Course in Fractional Sobolev Spaces},
 year = {2023}
}

@book{Rudin1976Principles,
  title={Principles of Mathematical Analysis},
  author={Rudin, Walter},
  series={International series in pure and applied mathematics},
  year={1976},
  publisher={McGraw-Hill}
}

@article{Shvartsman2009Sobolev,
author = {Pavel Shvartsman},
title = {Sobolev ${W}_p^1$-spaces on closed subsets of ${R}^n$},
journal = {Advances in Mathematics},
volume = {220},
number = {6},
pages = {1842-1922},
year = {2009}
}

@article{Shvartsman2017Whitney,
title = {Whitney-type extension theorems for jets generated by Sobolev functions},
journal = {Advances in Mathematics},
volume = {313},
pages = {379-469},
year = {2017},
author = {Pavel Shvartsman},
}

@book{Stein1970Singular,
 author = {Elias M. Stein},
 publisher = {Princeton University Press},
 title = {Singular Integrals and Differentiability Properties of Functions (PMS-30)},
 year = {1970}
}

@article{Whitney1934Analytic,
 author = {Hassler Whitney},
 journal = {Transactions of the American Mathematical Society},
 number = {1},
 pages = {63--89},
 publisher = {American Mathematical Society},
 title = {Analytic Extensions of Differentiable Functions Defined in Closed Sets},
 volume = {36},
 year = {1934}
}

@article{Whitney1934Differentiable,
 author = {Hassler Whitney},
 journal = {Transactions of the American Mathematical Society},
 number = {2},
 pages = {369--387},
 publisher = {American Mathematical Society},
 title = {Differentiable Functions Defined in Closed Sets. I},
 volume = {36},
 year = {1934}
}

@article{Whitney1934Functions,
 author = {Hassler Whitney},
 journal = {Annals of Mathematics},
 number = {3},
 pages = {482--485},
 publisher = {Annals of Mathematics},
 title = {Functions Differentiable on the Boundaries of Regions},
 volume = {35},
 year = {1934}
}

\end{document}